\documentclass[12pt]{amsart}
\usepackage{amsmath,amssymb,amsbsy,amsfonts,latexsym,amsopn,amstext,cite,
                                               amsxtra,euscript,amscd,bm,mathabx}
\usepackage{graphicx}
\usepackage{float} 
\usepackage[english]{babel}
\usepackage{mathtools}
\usepackage{geometry}
\usepackage[colorlinks,linkcolor=blue,anchorcolor=blue,citecolor=blue,backref=page]{hyperref}

\DeclareMathAlphabet{\mathmybb}{U}{bbold}{m}{n}

\numberwithin{equation}{section}

\newtheorem{thm}{Theorem}[section]
\newtheorem{lem}[thm]{Lemma}
\newtheorem{prop}[thm]{Proposition} 
 
\newtheorem{prob}[thm]{Problem}

\newcommand{\N}{\mathbb{N}}
\newcommand{\R}{\mathbb{R}}

\begin{document}

\title[Digits of partition functions]
{On the Digits of Partition Functions}

\author[Siddharth Iyer] {Siddharth Iyer}
\address{School of Mathematics and Statistics, University of New South Wales, Sydney, NSW 2052, Australia}
\email{siddharth.iyer@unsw.edu.au}

\begin{abstract}
We study a problem of Douglass and Ono concerning the smallest integer $n$ such that the partition function $p(n)$ begins with a specified string of digits $f$ in base $b$. By employing an elementary discrepancy framework, we establish new upper bounds that significantly improve upon previous results of Luca.
\end{abstract}


\maketitle

\section{Introduction}
For a natural number $n$, let $p(n)$ denote the number of integer partitions of $n$, and let $PL(n)$ denote the number of plane partitions of $n$. Recall that a plane partition of size $n$ is an array of non-negative integers $(\pi_{i,j})$ such that rows and columns are weakly decreasing and $\sum_{i,j} \pi_{i,j} = n$ (for more background, see \cite{Andrews}). 

Recently, it was established that the leading digits of both $p(n)$ and $PL(n)$ abide by Benford's Law \cite{Anderson, Ono}. In light of this, Douglass and Ono \cite{Ono} asked the following specific question (which we generalize here to include $p(\cdot)$):

\begin{prob}
Let $\mathbf{p} \in \{p(\cdot), PL(\cdot)\}$. For each string $f$ in base $b$ (not beginning with the digit $0$), let $N_{\mathbf{p}}(f,b)$ be the smallest positive integer with the property that $\mathbf{p}(N_{\mathbf{p}}(f,b))$ begins with the string $f$ in base $b$. Find non-trivial upper bounds for $N_{\mathbf{p}}(f,b)$.
\end{prob}

Using exponential sums and transcendence theory, Luca \cite{LucaPart2} proved that $N_{p}(f,b) \leq \exp(2\cdot 10^{25}(t+12)(\ln b)^2)$, where $t$ is the number of base $b$ digits of $f$. Similarly, using properties of prime numbers, he proved that $N_{PL}(f,b) \leq b^{51t+688}$ \cite{LucaPart1}. 

In this paper, we demonstrate that an elementary approach utilizing the mean value theorem and the fractional parts of logarithms yields drastically sharper bounds.

\begin{thm}
\label{PartitionDigitthm}
Let $b \geq 2$ be an integer base and let $t$ be the length of the digit string $f$. We have
\begin{equation*}
N_{p}(f,b) \leq 288 \cdot b^{2t}+2,
\end{equation*}
and
\begin{equation*}
N_{PL}(f,b) \leq 130 \cdot b^{3t} + 29400 \cdot b^{3t/2}.
\end{equation*}
\end{thm}

\section{Preliminaries}

For an integer $b \geq 2$ and $t \in \N$ (with $t \geq 2$ if $b = 2$), let $\mathcal{D}_{b,t}$ denote the set of integers of the form $\sum_{d=0}^{t-1}a_{d} b^{d}$, where $a_{i} \in \{0,\ldots,b-1\}$ and $a_{t-1} \neq 0$. Note that $f \in \mathcal{D}_{b,t}$ implies $b^{t-1} \le f \le b^t - 1$. 

Define the function $C_{b,t}: \{b^{t-1}, b^{t-1}+1,\ldots\} \rightarrow \mathcal{D}_{b,t}$ as follows: if $n = \sum_{d=0}^{z+t-1}a_{d} b^{d}$ for some integer $z \geq 0$ and $a_{z+t-1} \neq 0$, then 
\begin{equation*}
C_{b,t}(n) = \sum_{d=0}^{t-1}a_{d+z} b^{d}.
\end{equation*}
In other words, $C_{b,t}(n)$ extracts the leading $t$ digits of $n$ in base $b$. Let $\{x\}$ denote the fractional part of $x$.

\begin{lem}
\label{logdigitmod1}
Suppose that $n \geq b^{t-1}$. If $f \in \mathcal{D}_{b,t}$, then $C_{b,t}(n) = f$ if and only if 
\begin{equation*}
\{\log_{b}(n)\} \in [\log_{b}(f)-t+1, \log_{b}(f+1)-t+1).
\end{equation*}
\end{lem}
\begin{proof}
Note that $C_{b,t}(n) = f$ if and only if $n \in [fb^{z},(f+1)b^z)$ for some integer $z \geq 0$. Taking the base $b$ logarithm of this interval yields $z + \log_b(f) \le \log_b(n) < z + \log_b(f+1)$. Because $f \in \mathcal{D}_{b,t}$, we know $\log_b(f)$ is between $t-1$ and $t$. Setting $z = \lfloor \log_b(n) \rfloor - t + 1$ yields the required interval for the fractional part.
\end{proof}

For subsequent error bounds, we rely on the elementary inequality
\begin{equation}
\label{log1+xlemma}
|\ln(1+y)| \leq 2|y| \quad \text{for } |y| \leq 1/2.
\end{equation}

\begin{lem}
\label{logbpartitionfunc}
For $n \geq  4$, we have the inequality
\begin{equation*}
\left|\log_{b}(p(n)) - \left(\frac{\pi \sqrt{24}}{6 \ln b} \sqrt{n}-\frac{\ln n}{\ln b}+\log_{b}\left(\frac{\sqrt{3}}{12}\right) \right)\right| \leq \frac{4}{n^{1/2}\ln b}.
\end{equation*}
\end{lem}
\begin{proof}
Let $\mu(n) = \frac{\pi}{6}\sqrt{24n-1}$. According to \cite[Lemma 1]{LucaPart2}, for $n \geq 4$, we have
\begin{equation*}
\mu(n) + \ln\left(\frac{\sqrt{3}}{12n} -\frac{\sqrt{3}}{12n^{3/2}}\right) < \ln(p(n)) < \mu(n) + \ln\left(\frac{\sqrt{3}}{12n} +\frac{\sqrt{3}}{12n^{3/2}}\right).
\end{equation*}
Observe that $\ln\left(\frac{\sqrt{3}}{12n} \pm \frac{\sqrt{3}}{12n^{3/2}}\right) = \ln\left(\frac{\sqrt{3}}{12n}\right) + \ln\left(1\pm\frac{1}{n^{1/2}}\right)$. By inequality \eqref{log1+xlemma}, we have $|\ln(1\pm n^{-1/2})| \leq 2 n^{-1/2}$. Furthermore, note that 
\begin{equation*}
\frac{\pi \sqrt{24}}{6} \sqrt{n} - \frac{\pi}{6(\sqrt{24 n}+ \sqrt{24 n-1})} = \mu(n) < \frac{\pi \sqrt{24}}{6} \sqrt{n}.
\end{equation*}
This yields the estimate $\left|\mu(n) - \frac{\pi \sqrt{24}}{6} \sqrt{n}\right| \leq \frac{2}{\sqrt{n}}$. Combining these errors, we obtain:
\begin{equation*}
\left|\ln(p(n)) - \left(\frac{\pi \sqrt{24}}{6} \sqrt{n}-\ln(n)+\ln\left(\frac{\sqrt{3}}{12}\right) \right)\right| \leq \frac{4}{n^{1/2}}.
\end{equation*}
The lemma follows by dividing this inequality by $\ln b$.
\end{proof}

\begin{lem}
\label{logbplanepartitionalmeqn}
Let $A = \zeta(3)\approx 1.202$, $c = \zeta'(-1)\approx -0.165$, and $B = \frac{2^{25/26}e^{c}A^{7/26}}{\sqrt{12\pi}}$. For $n \geq 2829$, we have
\begin{equation*}
\left|\log_{b}(PL(n)) - \left(\frac{3(A/4)^{1/3}}{\ln b}n^{2/3}-\frac{25}{36 \ln b}\ln n + \log_{b}(B)\right)\right| \leq \frac{200}{n^{2/3}\ln b}.
\end{equation*}
\end{lem}
\begin{proof}
Let $\mu(n) = 3(A/4)^{1/3}n^{2/3}$. According to \cite[Lemma 1]{LucaPart1}, when $n \geq 1001$, we have
\begin{equation*}
\ln\left(\frac{B}{n^{25/36}}e^{\mu(n)}\right)+\ln\left(1-\frac{100}{n^{2/3}}\right)<\ln(PL(n)) <\ln\left(\frac{B}{n^{25/36}}e^{\mu(n)}\right)+\ln\left(1+\frac{100}{n^{2/3}}\right).
\end{equation*}
By applying \eqref{log1+xlemma}, for $n \geq 2829$ (which ensures $100n^{-2/3} \le 1/2$), we obtain:
\begin{equation*}
\left|\ln(PL(n)) - \left(3(A/4)^{1/3}n^{2/3}-\frac{25}{36}\ln n + \ln B\right)\right| \leq \frac{200}{n^{2/3}}.
\end{equation*}
Dividing this by $\ln b$ completes the proof.
\end{proof}

\section{A General Discrepancy Bound}

To bound $N_{\mathbf{p}}(f,b)$, we formalize a framework for finding the smallest integer within a specific interval whose fractional part lands in a target range $[a, a+\delta)$. 

\begin{prop}
\label{prop:framework}
Let $g: [K, \infty) \rightarrow \R$ be a function of the form 
\begin{equation*}
g(x) = h(x) + E(x) = c_{1}x^{\theta}+ c_{2}\ln x+c_{3}+E(x),
\end{equation*}
where $c_{1} > 0$, $c_{2} \leq 0$, $c_{3}\in \R$, $0<\theta<1$, and $|E(x)| \leq c_{4}x^{-\theta}$ for some constant $c_{4}>0$. Let $\delta > 0$ and $[a, a+\delta) \subseteq [0,1)$. Define the following constants:
\begin{align*}
L_1 &= \left(\frac{-3c_{2}}{c_{1}\theta}\right)^{1/\theta}, & L_2 &= \left(\frac{3 c_{4}}{\delta}\right)^{1/\theta}, & L_3 &= \left(\frac{2}{c_1 \theta 2^{\theta-1}}\right)^{1/\theta}, \\
L_4 &= \left(\frac{3 c_{1} \theta}{\delta} \right)^{\frac{1}{1-\theta}}, & D &= \frac{2}{c_{1}\theta 2^{\theta-1}}. & &
\end{align*}
Then, there exists an integer $m$ satisfying $\{g(m)\} \in [a,a+\delta)$ such that
\begin{equation*}
m \leq 2\max\{K, L_{1}, L_{2}+1, L_{3}, L_{4}\}.
\end{equation*}
\end{prop}
\begin{proof}
By differentiating $h(x)$, we obtain $h'(x) = c_{1}\theta x^{\theta-1} + c_{2}/x$. For $x \geq L_{1}$, it is straightforward to check that $h'(x) \geq \frac{2}{3}c_{1}\theta x^{\theta-1}$. Furthermore, for $x \geq 1$, we clearly have $h'(x) \leq c_{1}\theta x^{\theta-1}$ since $c_2 \leq 0$. 

Notice that when $x \geq L_{3} = D^{1/\theta}$, we have $Dx^{1-\theta} \leq x$, which implies the interval $[x, x+Dx^{1-\theta}]$ is contained within $[x, 2x]$. For $x \geq \max\{L_1, L_3\}$ we have
\begin{align*}
h(x+Dx^{1-\theta}) - h(x) &= \int_{x}^{x+Dx^{1-\theta}} h'(y) \, dy \\
&\geq \frac{2}{3} \int_x^{x+Dx^{1-\theta}} c_1 \theta y^{\theta-1} \, dy \\
&= \frac{2}{3} c_1 \left( (x+Dx^{1-\theta})^\theta - x^\theta \right).
\end{align*}
Applying the Mean Value Theorem, the difference $(x+Dx^{1-\theta})^\theta - x^\theta$ evaluates to $\theta C^{\theta-1} Dx^{1-\theta}$ for some $C \in [x, 2x]$. Since $\theta - 1 < 0$, the minimum is attained at $C = 2x$. Thus:
\begin{equation*}
h(x+Dx^{1-\theta}) - h(x) \geq \frac{2}{3} c_1 \theta (2x)^{\theta-1} D x^{1-\theta} = \frac{2}{3} c_1 \theta 2^{\theta-1} D = \frac{4}{3}.
\end{equation*}
Because $4/3 > 1 + \delta/3$, the interval $[h(x), h(x+Dx^{1-\theta})]$ is sufficiently large to choose values $y_{1} < y_{2}$ such that $y_1 \equiv a + \delta/3 \pmod 1$ and $y_{2} = y_{1} + \delta/3$. By continuity, there exist $x_1 < x_2$ in $[x, x+Dx^{1-\theta}]$ such that $h(x_i) = y_i$. 

Using the Mean Value Theorem once more, $y_2 - y_1 = h'(c)(x_2 - x_1)$ for some $c \in [x_1, x_2]$. Because $h'(c) \leq c_1 \theta c^{\theta-1} \leq c_1 \theta x^{\theta-1}$, we find:
\begin{equation*}
x_{2}-x_{1} \geq \frac{\delta/3}{c_1 \theta x^{\theta-1}} = \frac{\delta}{3 c_{1}\theta}x^{1-\theta}.
\end{equation*}
If we require $x \geq L_{4}$, then $x_{2}-x_{1} \geq 1$. Consequently, the interval $[x_{1},x_{2}]$ must contain an integer $m$. Because $h(m) \in [y_1, y_2]$, we know $\{h(m)\} \in [a+\delta/3, a+2\delta/3]$. Finally, if $x > L_2$, we guarantee that the error term satisfies $|E(m)| \leq c_4 x^{-\theta} < \delta/3$. Therefore, $\{g(m)\} = \{h(m) + E(m)\} \in (a, a+\delta)$.

Taking $x = \max\{K, L_1, L_2+1, L_3, L_4\}$, the integer $m$ constructed above satisfies $m \in [x,x+Dx^{1-\theta}]$, and thus $m \leq 2x$.
\end{proof}

\section{Proof of Theorem \ref{PartitionDigitthm}}

To prove Theorem \ref{PartitionDigitthm}, we apply Proposition \ref{prop:framework} in conjunction with Lemma \ref{logdigitmod1}. For a target string $f \in \mathcal{D}_{b,t}$, the required fractional part interval has width $\delta = \log_b(f+1) - \log_b(f) = \frac{\ln(1+1/f)}{\ln b}$. 

Using the standard inequality $\ln(1+y) \geq \frac{y}{1+y}$ for $y > 0$, and noting that the maximum value of $f \in \mathcal{D}_{b,t}$ is $b^t - 1$, we obtain the lower bound:
\begin{equation}
\label{eq:delta_bound}
\delta \geq \frac{1}{(f+1)\ln b} \geq \frac{b^{-t}}{\ln b}.
\end{equation}
Because the upper bound for $m$ in Proposition \ref{prop:framework} is non-decreasing as $\delta$ decreases, we substitute $\delta = \frac{b^{-t}}{\ln b}$ into our bounds to isolate the worst-case scenario.

\subsection{Bounds for the Partition Function $p(n)$}
Applying Lemma \ref{logbpartitionfunc}, we have $\theta = 1/2$, $c_{1} = \frac{\pi \sqrt{24}}{6\ln b}$, $c_{2} = \frac{-1}{\ln b}$, and $c_{4} = \frac{4}{\ln b}$ for $K=4$. We compute the constants from Proposition \ref{prop:framework}:
\begin{align*}
L_{1} &= \left(\frac{36}{\pi\sqrt{24}}\right)^{2} = \frac{54}{\pi^2} \approx 5.47, \\
L_{2} &= \left(\frac{12/\ln b}{b^{-t}/\ln b}\right)^2 = 144 b^{2t}, \\
L_{3} &= \left(\frac{4 \sqrt{3}}{\pi} \ln b\right)^2 = \frac{48}{\pi^2}(\ln b)^2, \\
L_{4} &= \left(\frac{3 (\pi \sqrt{24} / 12 \ln b)}{b^{-t}/\ln b}\right)^2 = \frac{3\pi^2}{2} b^{2t} \approx 14.8 b^{2t}.
\end{align*}
Since $t \ge 1$ and $b \ge 2$, the maximum over all constants is solidly dominated by $L_2$. Thus, $N_p(f,b) \leq 2(144 b^{2t}+1) = 288 b^{2t}+2$.

\subsection{Bounds for the Plane Partition Function $PL(n)$}
Applying Lemma \ref{logbplanepartitionalmeqn}, we have $\theta = 2/3$, $c_1 = \frac{3(A/4)^{1/3}}{\ln b}$, $c_2 = \frac{-25}{36 \ln b}$, and $c_4 = \frac{200}{\ln b}$ for $K=2829$. Computing the constants:
\begin{align*}
L_{1} &= \left(\frac{125}{24\sqrt{6}\sqrt{A}}\right)^{3/2} \approx 1.94, \\
L_{2} &= \left(\frac{600/\ln b}{b^{-t}/\ln b}\right)^{3/2} = 600^{3/2} b^{3t/2} \approx 14697 b^{3t/2}, \\
L_{3} &= \left(\frac{2 \ln b}{A^{1/3}}\right)^{3/2} \approx 2.4 (\ln b)^{3/2}, \\
L_{4} &= \left(\frac{6 (A/4)^{1/3} / \ln b}{b^{-t}/\ln b}\right)^3 = 54A b^{3t} \approx 64.9 b^{3t}.
\end{align*}
Here, the dominant bounds are $L_2 = O(b^{3t/2})$ and $L_4 = O(b^{3t})$. When $b^t$ is small, $L_2$ is the maximum, but as $b^t$ grows, $L_4$ rapidly overtakes it. Taking the maximum safely bounded by the sum, we find $N_{PL}(f,b) \le 2(L_4 + L_2)$, which gives the final bound:
\begin{equation*}
N_{PL}(f,b) \leq 130 \cdot b^{3t} + 29400 \cdot b^{3t/2}.
\end{equation*}
This formally proves both bounds of Theorem \ref{PartitionDigitthm}.

\section*{Acknowledgments}
The author would like to thank Igor Shparlinski for suggesting the paper \cite{LucaPart2}. This research was supported by the Australian Government Research Training Program (RTP) Scholarship and a top-up scholarship from the University of New South Wales, both of which were instrumental in this work.
\section*{Data availability statement}
Data sharing not applicable to this article as no datasets were generated or analysed during the current study.
\section*{Conflict of Interest}
There is no Conflict of interest.
 
\end{document}